\documentclass[11pt]{amsart}
\usepackage{amsxtra}
\usepackage{amssymb,mathrsfs}
\usepackage{mathtools,yhmath}
\usepackage{color}
\usepackage[all]{xy}
\usepackage{tikz,graphicx}
\theoremstyle{plain}

\newcommand{\cleqn}{\setcounter{equation}{0}}
\newcommand{\clth}{\setcounter{theorem}{0}}
\newcommand {\sectionnew}[1]{\section{#1}\cleqn\clth}
\newcommand{\nn}{\hfill\nonumber}
\newtheorem{theorem}{Theorem}[section]
\newtheorem{lemma}[theorem]{Lemma}
\newtheorem{definition-theorem}[theorem]{Definition-Theorem}
\newtheorem{proposition}[theorem]{Proposition}
\newtheorem{corollary}[theorem]{Corollary}
\newtheorem{definition}[theorem]{Definition}
\newtheorem{example}[theorem]{Example}
\newtheorem{remark}[theorem]{Remark}
\newtheorem{conjecture}[theorem]{Conjecture}

\newcommand \bth[1] { \begin{theorem}\label{t#1} }
\newcommand \ble[1] { \begin{lemma}\label{l#1} }

\newcommand \bpr[1] { \begin{proposition}\label{p#1} }
\newcommand \bco[1] { \begin{corollary}\label{c#1} }
\newcommand \bde[1] { \begin{definition}\label{d#1}\rm }
\newcommand \bex[1] { \begin{example}\label{e#1}\rm }
\newcommand \bre[1] { \begin{remark}\label{r#1}\rm }
\newcommand \bcj[1] { \begin{conjecture}\label{j#1}\rm }

\renewcommand {\eth} { \end{theorem} }
\newcommand {\ele} { \end{lemma} }

\newcommand {\epr} { \end{proposition} }
\newcommand {\eco} { \end{corollary} }
\newcommand {\ede} { \end{definition} }
\newcommand {\eex} { \end{example} }
\newcommand {\ere} { \end{remark} }
\newcommand {\ecj} { \end{conjecture} }

\newcommand {\enota} { \end{notation} }
\newcommand \thref[1]{Theorem \ref{t#1}}
\newcommand \leref[1]{Lemma \ref{l#1}}

\newcommand \deref[1]{Definition \ref{d#1}}
\newcommand \exref[1]{Example \ref{e#1}}
\newcommand \reref[1]{Remark \ref{r#1}}
\newcommand \lb[1]{\label{#1}}
\newtheorem*{maintheorem*}{Main Theorem}
\newtheorem*{theorem*}{Theorem}
\newtheorem*{theoremA*}{Theorem A}
\newtheorem*{theoremB*}{Theorem B}
\newtheorem{thmintro}{Theorem}

\def \Cset {{\mathbb C}}
\def \KK {{\mathbb K}}
\def \Zset {{\mathbb Z}}
\def \Nset {{\mathbb N}}
\def \Qset {{\mathbb Q}}

\def \B  {{\widetilde{B}}}
\newcommand \ex {{\bf{ex}}}
\newcommand \fr {{\bf{fr}}}

\def \de {\delta}

\def \la {\lambda}

\def \Om {\Omega}

\def \de {\delta}

\def \sig {\sigma}

\def \mt  {\mapsto}
\def \ci  {\circ}           % duals
\def \bu {\bullet}

\def \wh {\widehat}

\def \deg { {\mathrm{wt}} }
\def \Id { {\mathrm{Id}} }
\def \id { {\mathrm{id}} }

\def \Lie { {\mathrm{Lie \,}} }

\DeclareMathOperator \Span { {\mathrm{Span}} }

\DeclareMathOperator \range {\mathrm{range}}

\DeclareMathOperator \diag { {\mathrm{diag}} }

\DeclareMathOperator \MaxSpec { {\mathrm{MaxSpec}} }

\renewcommand \max { {\mathrm{max}} }

\newcommand*\circled[1]{\tikz[baseline=(char.base)]{
            \node[shape=circle,draw,inner sep=2pt] (char) {#1};}}
\begin{document}
\title
[Maximal green sequences]
{Maximal green sequences for \\ Quantum and Poisson CGL extensions}
\author[Milen Yakimov]{Milen Yakimov}
\address{
Department of Mathematics, Northeastern University, Boston, MA 02115 and International Center for Mathematical Sciences, Institute of Mathematics and Informatics \\
Bulgarian Academy of Sciences \\ 
Acad. G. Bonchev Str., Bl. 8 \\
Sofia 1113, Bulgaria}
\email{m.yakimov@northeastern.edu}
\dedicatory{Dedicated to Dan Nakano on his 60th birthday with admiration}
\begin{abstract} We prove that the quantum and classical cluster algebras for all members of the axiomatically defined classes of symmetric quantum and Poisson Cauchon--Goodearl--Letzter extensions possess maximal green sequences in the sense of Keller. Previously, maximal green sequences were constructed for explicit families of cluster algebras; many of those can be recovered from the general result for CGL extensions. 
\end{abstract}
\maketitle
%%%%%%%%%%%%%%%%%%%%%%%%%%%%%%%%%%%%%%%%%%%%%%%%%%%%%%%%%%%%%%%%%%%%%%%%%%%
\sectionnew{Introduction}
\lb{Intro}
\subsection{Setting}
\label{1.1}
Cluster Algebras were invented by Fomin and Zelevinsky in 2001 and were envisioned as a tool for the qualitative study of canonical bases of quantum groups and total positivity in algebraic varieties \cite{Fo}. Within a short time, they found a broad range of applications in many areas of mathematics and mathematical physics. Maximal green sequences were invented by Keller in 2010 and were implicitly used by Gaiotto, Moore and Neitzke in \cite{GMN}. They have a simple combinatorial definition: A maximal green sequence is a sequence of mutations of an exchange matrix at green indices (i.e., having non-negative $c$-vectors) is performed until all indices become red (i.e., having non-positive $c$-vectors). A reddening sequence (also called a green-to-red sequence) is a more general notion of a sequence of mutations at both green and red indices until all vertices become red. These notions have a number of applications: 
\begin{enumerate}
\item Keller \cite{K1} proved that pairs of reddening sequences with the same initial and end seeds yield dilogarithm identities.
\item Gross, Hacking, Keel and Kontsevich \cite{GHKK} proved that if a cluster algebra possesses a reddening sequence, then the Fock--Goncharov duality conjectures \cite{FG} hold for it. So, the corresponding upper cluster algebra has a basis parametrized by the tropical points of a cluster Poisson variety. 
\item Existence of a reddening sequence is equivalent to injective reachability, which plays a key role in Qin's construction of common triangular bases \cite[Sect. 2.3]{Q}.
\end{enumerate}
Several general properties of the two notions were obtained, which include the following:
\begin{enumerate}
\item[(1')] Muller \cite{M} proved that the existence of a reddening sequence is preserved under mutation, but the same is not true for maximal green sequences. 
\item[(2')] Br\"ustle, Hermes, Igusa and Todorov \cite{BHIT} proved that maximal green sequences remain such under cyclic permutations. They also proved several key conjectures, including a finiteness result on the number of maximal green sequences for a valued quiver that is mutation equivalent to an acyclic quiver. 
\item [(3')] Br\"ustle, Dupont and P\'erotin \cite{BDP} proved that the end quiver of a maximal green sequence is isomorphic to the original under some permutation of its vertices.  
\end{enumerate}
We refer the reader to Keller's comprehensive review \cite{K3} for many additional results. 
\medskip

Due to the importance of maximal green and reddening sequences, a large body of research was done on their construction for many concrete families of cluster algebras:
\begin{enumerate}
\item[(a)] mutation finite quivers and Teichmuller theory \cite{B,BMi,GS,Mi},
\item[(b)] double Bruhat cells and Grassmannians \cite{W1,W2}.
\item[(c)] double Bott--Samelson cells \cite{SW}. 
\end{enumerate}

Maximal green sequences for $A_n$ quivers were classified in \cite{GM}.
The existence property for maximal green and reddening sequences was shown to be inherited under taking subquivers \cite{M}. A reverse implication that a triangular extension of quivers with maximal green sequences has the same property was proved in \cite{CL}. Gluing procedures were obtained in \cite{BM,BMRYZ}. 

\subsection{CGL extensions} 
Symmetric quantum and Poisson Cauchon--Goodearl--Letzter (CGL) extensions are large axiomatically defined classes of noncommutative and Poisson algebras that have PBW bases 
$\{ x_1^{m_1} \ldots x_N^{m_N} \mid m_1, \ldots, m_N \in \Nset \}$ such that for all $j <k$,
\begin{align*}
x_k x_j &\in \la_{kj} x_j x_k + \Span \{  
x_{j+1}^{m_{j+1}} \ldots x_{k-1}^{m_{k-1}}
\mid m_{j+1}, \ldots, m_{k-1} \} \quad \mbox{and} \\
\{x_k, x_j\} &\in \la_{kj} x_j x_k + \Span \{  
x_{j+1}^{m_{j+1}} \ldots x_{k-1}^{m_{k-1}}
\mid m_{j+1}, \ldots, m_{k-1} \}, 
\end{align*}
respectively, for certain scalars $\lambda_{kj}$ coming from a torus action.

The two classes of CGL extensions appear to be very large. No classification results are known except for very low dimensions. It was proved in \cite{LM} that the Poisson CGL extensions coming from Lie theory only cover a small portion of the class of Poisson CGL extensions. Section \ref{prelim-CGL} contains a brief review of the theory of CGL extensions.

In \cite{GY0,GY-Memo,GY-memo2} it was proved that every symmetric quantum or Poisson CGL extension satisfying mild assumptions possesses a quantum (classical) cluster algebra structure, which coincides with the corresponding upper quantum (classical) cluster algebra, see Theorems \ref{tcluster} and \ref{tcluster-P}. 
Our main result is that every member of these large axiomatic classes of cluster algebras possesses (multiple) maximal green sequences.

\begin{thmintro} 
\label{thmA}
Each quantum (classical) cluster algebra on a symmetric quantum (Poisson) CGL extension of dimension $N$ possesses maximal green sequences parametrized by all reduced expressions $\boldsymbol{w}$ of the longest element of the symmetric group $S_N$ with the property that its  initial subwords are such that 
$\boldsymbol{w}_{\leq k}(\{1,\ldots,j\})$ is an interval in $\{1, \ldots, N\}$ for all $2 \leq j,k \leq N-1$.  
\end{thmintro}

%Since the initial and end seeds of the maximal green sequences in the theorem are the same, all symmetric CGL extensions give dilogarithm identities by Keller's method \cite{K1}. 

The theorem implies the existence of maximal green sequences for many explicit families of (quantum) cluster algebras that have attracted strong interest:
\begin{enumerate}
    \item The integral forms of the quantum unipotent cells of all quantized universal enveloping algebras of symmetrizable Kac-Moody algebras are integral forms of  symmetric quantum CGL extensions \cite{GY2}.  
    \item The quantum double Bruhat cells of all quantized coordinate rings of complex simple algebraic groups are symmetric quantum CGL extensions \cite{GY-BZ}.
    \item The coordinate rings of all full flag varieties and all flag varieties of type $A$ are obtained from the cluster structures in item (1) by changes in the frozen variables \cite{Fr,GLS0}.
    \item The (integral forms) of the (quantized) coordinate rings of complex simple algebraic groups are obtained from the (upper) cluster structures in item (2) by changes in the frozen variables \cite{O,OQY,QY}.
    \item All Bott--Samelson cells are symmetric Poisson CGL extensions \cite{EL}.  
    \item The Poisson geometry of mixed Poisson structures (including double Bott--Samelson varieties) was studied in detail in \cite{LMo}. It is very likely that they lead to large families of symmetric Poisson CGL extensions. 
\end{enumerate}

Theorem \ref{thmA} is proved in Section \ref{resCGL} and is  derived from a more general result on sequences of mutations, which we call {\em{full layered $T$-systems}}. These are sequences of mutations of valued quivers whose sets of vertices $S$ are totally ordered and partitioned, and have the properties that (1) at each step the incoming arrows to the mutation vertex are from the immediate predecessor and successor in the same stratum and (2) the mutation steps at the vertices of a given stratum $S'$ are $|S'| (|S'| -1)/2$ and are organized in a certain $A$-type mutation pattern, see Definitions \ref{dTsyst} and \ref{dTsyst2}. Here and below $|S'|$ denotes the cardinality of a finite set. 

\begin{thmintro} 
\label{thmB}
Every full layered $T$-system is a maximal green sequence. 
\end{thmintro}

Theorem \ref{thmB} is proved in Section \ref{Tsystres}.
\medskip

\noindent
{\bf{Remark.}}

(1) It follows from \thref{ind-Tsyst} and \reref{restr} that (i) each full layered $T$-system is a component preserving sequence of mutations in the sense of Bucher, Machacek, Runburg, Yeck and Zewde \cite{BMRYZ}, (ii) at each step the valued quiver is a union of Dynkin graphs of type $A$ with some orientations and various additional valued arrows between them and (iii) the layered $T$-system is itself a shuffle of $A$-type mutation sequences. Note that the collection of quivers with the described properties in (ii) is very general and contains all quivers for double Bruhat cells \cite{BFZ05}; in fact, if one takes the finest partition, all valued quivers arise this way.

(2) Theorem \ref{thmB} is not a consequence of \cite[Theorem 3.11 and Corollary 3.12]{BMRYZ} for two reasons. Firstly, the definition of component preserving mutation sequence in \cite{BMRYZ} requires (*) knowledge of the arrows between the mutation vertex at each step and the frozen vertices. In the definition of layered $T$-system, we only have (**) assumptions on arrows between exchangeable vertices. The properties (*) are only obtained in the proof of \thref{ind-Tsyst} which is a necessary stronger form of Theorem \ref{thmA} for an inductive argument. Secondly, we do not assume anything about the full subquivers (for the strata of the partition of the vertices) of the valued quivers at each step of the mutation. This necessary property for the verification of the other part of the definition of component preserving mutation sequence in \cite{BMRYZ} is again only obtained in \thref{ind-Tsyst}. 

(3) The application of \cite[Theorem 3.11 and Corollary 3.12]{BMRYZ} seems difficult because of the required knowledge about arrows between the mutation vertex and frozen vertices, which is essentially performing a sequence of mutations on a framed quiver and checking the conditions for a maximal green sequence. At the same time, Theorem \ref{thmB} is easy to apply since it deals with the original quiver $Q$ and not its framed quiver $\widehat{Q}$. For instance, Keller's conjecture \cite{K3} that the mutation sequences of Gei\ss--Leclerc--Schr\"oer for certain seeds of the cluster algebra structure on the coordinate ring of every unipotent cell of a symmetric Kac--Moody algebra \cite{GLS1} can be proved by a short application of Theorem \ref{thmB}, see \exref{GLSseq}.

(4) In order to obtain Theorem \ref{thmA} from the results in \cite{GY0,GY-Memo,GY-memo2}, we have to perform a surgery on frozen variables of cluster algebras. The results of \cite{GY0,GY-Memo,GY-memo2} give information about cluster structures on CGL extensions that have very particular choices of frozen variables. Those results are proved by using noncommutative and Poisson algebra methods (noncommutative and Poisson UFDs). This information is converted to combinatorial information for input to Theorem \ref{thmA} by removing all frozen indices. Then principal coefficients are inserted by applying the combinatorial Theorem \ref{thmA}.
\medskip

Given a symmetric quantum CGL extension $R$, we can consider the localized quotients $(R/I)[E^{-1}]$ where $I$ is a torus invariant prime ideal and $E$ is the collection of nonzero homogeneous normal elements of the factor algebra. On the Poisson side, we can consider such quotients where $I$ is the vanishing ideal of a torus orbit of symplectic leaves of the spectrum of $R$ and $E$ is the collection of nonzero Poisson-normal elements. Clusters of such algebras can be built along the lines of collecting sequences of normal elements in the images of saturated chains of subalgebras of $R$ as was done for quantum Richardson varieties for symmetrizable Kac--Moody groups in \cite{LY}. We conjecture that the process of reordering of the generators and recomputing normal elements as in Theorem \ref{thmA} is a reddening sequence for these even more general situations than those in Theorem A. 

\subsection*{Notation} For positive integers $j \leq k$, set $[j,k]:= \{j, j+1, \ldots, k \}$. For a finite index set $I$ and a ring $A$, denote by $M_{I,I'}(A)$ 
the abelian group of matrices with entries in $A$ and rows and columns indexed by $I$ and $I'$, respectively. When $I=I'=[1,N]$, we will denote the group by $M_{N,N}(A)$ as usual. 

The one-line notation for the elements of the symmetric group $S_N$ will be written in square brackets to avoid confusion with sequences of mutations. So, 
\[
\sig = [k_1 \, k_2 \, \ldots \, k_N]
\]
is the permutation $\sig \in S_N$ such that $\sig(i) = k_i$ for all $1 \leq i \leq N$. Denote the simple transpositions 
\[
s_k := (k \, (k+1)) \quad \mbox{for}
\quad 1 \leq k \leq N-1.
\]
For a reduced expression $\boldsymbol{w} = s_{k_1} \ldots s_{k_l}$ of a permutation, denote the initial subwords
\[
\boldsymbol{w}_{\leq i} := s_{k_1} \ldots s_{k_i} \quad \mbox{for} \quad 0 \leq i \leq l.
\]
\subsection*{Acknowledgments} The research of the author was supported by the Bulgarian Science Fund grant KP-06-N 62/5 and the Ministry of Education and Science grant DO1-239/10.12.2024, the Simons Foundation grant SFI-MPS-T-Institutes-00007697 and the USA National Science Foundation grant DMS-2200762.
%%%%%%%%%%%
\sectionnew{Preliminaries on cluster algebras and maximal green sequences}
\lb{prelim-CA}
In this section we review background material on mutations of exchange matrices and valued ice quivers, as well as on maximal green sequences. 
\subsection{Combinatorial background on cluster algebras}
\label{2.1}
Cluster algebras of geometric type \cite{FZ1} and quantum cluster algebras \cite{BerZe} are defined starting from a rectangular integer matrix with skew-symmetrizable principal part. Fix a finite (index) set $I$ and two complementary (disjoint) subsets 
\[
I = \ex \sqcup \fr.
\]
The indices in $\ex$ will be called {\em{exchangeable}} (or {\em{mutable}}) and those in $\fr$
{\em{frozen}}. For a matrix 
\[
\B = (b_{ij}) \in M_{I, \ex}(\Zset),
\]
its $\ex \times \ex$ submatrix will be denoted by $B$ and will be called the principal part of $\B$. An {\em{exchange matrix}} $\B \in M_{I, \ex}(\Zset)$ is a matrix whose principal part $B$ is skew-symmetrizable; 
that is, there is a diagonal matrix $D=\diag(d_j,~j\in\ex)$ with diagonal entries $d_j \in \Zset_+$ such that $(DB)^T= DB$. The entries $d_j$ will be assumed to be relatively prime. 

The cluster algebra of geometric type \cite{FZ1} and the quantum cluster algebra \cite{BerZe}, associated to an exchange matrix $\B$ are certain commutative and noncommutative algebras, respectively, defined in an iterative way. In the quantum case, one also needs to assume the existence of a skew-symmetric $I \times I$ matrix compatible in certain way with $\B$, \cite{BerZe}. 
Cluster algebras are defined via consecutive mutations on seeds consisting of an exchange matrix and cluster variables \cite{FZ1}, while quantum cluster algebras are defined via consecutive mutations on quantum seeds consisting of an exchange matrix and a toric frame \cite{BerZe}.
The iterative combinatorial construction of \cite{FZ1,BerZe} uses a mutation procedure on the levels of cluster variables and exchange matrices. The {\em{mutation}} of the matrix $\B$ in the direction of $k \in \ex$ is defined by 
\[
\mu_k(\B) = (b_{ij}'):=
\begin{cases}
-b_{ij} & \text{if } i = k \text{ or } j=k \\
b_{ij} + \frac{|b_{ik}|b_{kj} + b_{ik}|b_{kj}|}{2} & \text{otherwise.}
\end{cases}
\]
Sometimes we will display the mutation process as $\B \stackrel{k}{\longrightarrow} \B'$. 

If the principal part of an exchange matrix $\B$ is skew-symmetric, one can encode it by an {\em{ice quiver}} on $N$ vertices, which is a directed graph $Q(\B)$ such that   
\begin{enumerate}
\item $Q(\B)$ has no loops and 2-cycles and  
\item $Q(\B)$ has no arrows between its frozen vertices (labeled by $\fr$). 
\end{enumerate}
The quiver $Q(\B)$ is defined by requiring that the number of arrows from vertex $i$ to vertex $j$ equal $b_{ij}$ 
if $b_{ij}>0$ and $-b_{ij}$ otherwise. All ice quivers arise in this way. 

The mutation of matrices is translated to a mutation process on ice quivers. The mutation of $Q(\B)$ at the vertex $k$ (which coincides with $Q(\mu_k(\B))$ is performed in 3 steps:
\begin{enumerate}
\item reverse all arrows to and from the vertex $k$; 
\item complete the 2-paths through the vertex $k$ for which at least one end vertex is not frozen to triangles (3 cycles); 
\item cancel out pairs of opposite arrows.    
\end{enumerate}
For example, the mutation at the vertex 3 of the following quiver without frozen vertices
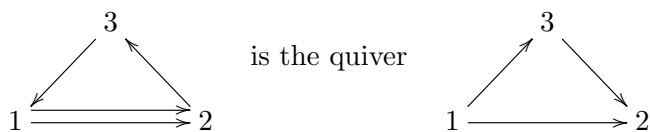
\[
\xymatrix{ & 3 \ar[ld] \\
1 \ar@<1ex>[rr] \ar@<0ex>[rr] & & 2 \ar[ul]
}
\quad  
\begin{matrix}
\\ \\
\mbox{is the quiver}
\end{matrix} \quad
\xymatrix{ & 3  \ar[dr] \\
1 \ar[ur] \ar@<0ex>[rr] & & 2
}
\]

General exchange matrices are encoded by {\em{valued ice quivers}} \cite[Sect. 3.3]{K2}, which are quivers whose vertices are indexed by $I$, satisfying the following conditions:
\begin{enumerate}
\item between each two vertices there is at most one arrow and there are no arrows between two frozen vertices (labeled by
$\fr$);
\item each arrow $i \to j$ is decorated with a pair of positive integers $(v_{ij}, v_{ji})$ for which there exist 
positive, relatively prime integers $d_i \in \Zset_+$ for $i \in \ex$ such that $d_i v_{ij} = d_j v_{ji}$ for all $i, j \in I$ connected by an arrow. Furthermore, if $i \in \fr$ or $j \in \fr$, then $v_{ij} = v_{ji}$. 
\end{enumerate}
If $v_{ij}= v_{ji}=1$, then the corresponding label is omitted. An ordinary ice quiver is transformed to a valued ice quiver by converting each $b$ arrows from $i$ to $j$ to $i \stackrel{(b,-b)}{\longrightarrow} j$. 

Let $\B$ be an exchange matrix. Set $b_{ji} := - b_{ij}$ for $i \in \fr$ and $j \in \ex$. The valued ice quiver $Q(\B)$ corresponding to $\B$ is the quiver for which there is an arrow $i \to j$ if $b_{ij}>0$ and its label is $(b_{ij}, -b_{ji})$. The collection of positive integers $\{d_i \mid i \in \ex\}$ for $Q(\B)$ is the same as the one for $\B$. The mutation of exchange matrices corresponds to a mutation of valued ice quivers, see \cite[Sect. 3.3]{K2}. 

By a valued quiver we mean a valued ice quiver without frozen vertices.
\subsection{Maximal green sequences}
Denote $\ex':=\{i' \mid i \in \ex\}$.
Let $B \in M_{\ex \times \ex}(\Zset)$
be the exchange matrix of a cluster algebra without frozen variables or the principal part of an exchange matrix $\B \in M_{I \times \ex}(\Zset)$. 
The corresponding cluster algebra with principal coefficients is associated to the exchange matrix 
\[
\wh{B}:= 
\begin{bmatrix}
B \\
\Id  
\end{bmatrix} \in M_{\ex \sqcup \ex', \ex}(\Zset)
\]
for the identity matrix $\Id := (c_{ij'})_{i,j \in \ex}$ where $c_{ij} = \delta_{ij}$.
 The corresponding valued ice quiver $Q(\wh{B})$ is the framed quiver of the quiver $Q(B)$, which is the quiver $\wh{Q(B)}$ obtained from the quiver $Q(B)$ by adding the vertices $\ex'$
and an edge $i \to e'$ for all $i \in \ex$. For example, the framed quiver of the $A_3$ type quiver $1 \leftarrow 2 \leftarrow 3$ is the quiver 
\begin{equation}
\label{frameA3}
\xymatrix{
1 \ar[d] & 2 \ar[l] \ar[d]
& 3 \ar[l] \ar[d] 
\\
1' & 2' & 3'
}
\end{equation}

Consider a sequence of mutations
\[
B' := \mu_{k_i} \ldots \mu_{k_1} (B).
\]
The $c$-vector of the index $j \in \ex$ is the second half of the $j$-th column of the matrix
\[
\mu_{k_i} \ldots \mu_{k_1} (\widehat{B})
\]
considered as a vector in $\Zset^{\ex'} \cong \Zset^\ex$. It depends on the sequence of mutations, not on the matrix $B'$ alone. Fomin and Zelevinsky \cite{FZ4} conjectured that the coordinates of each $c$-vector are either all non-negative or all non-positive (sign coherence conjecture for $c$-vectors).
%, in which case the vertex is called green or red, respectively. 
The conjecture was proved in the skewsymmetric case in \cite{DWZ,N,P} and, in general, by Gross, Hacking, Keel and Kontsevich \cite{GHKK}. 
\bde{green} (Keller, \cite{K1})
\hfill 
\begin{enumerate}
\item[(a)] The vertex $i \in \ex$ is a {\em{green}} (respectively, {\em{red}}) vertex of $\mu_{k_i} \ldots \mu_{k_1} (B)$ if all coordinates of its $c$-vector are non-negative (respectively, non-positive). 
\item[(b)] A sequence of mutations $k_1, \ldots, k_n$ is a {\em{green sequence}} for the 
exchange matrix $B$ if for all $1 \leq i \leq n$, $k_i$ is a green vertex of $\mu_{k_{i-1}} \ldots \mu_{k_1} (B)$. 
\item[(c)] A sequence of mutations $k_1, \ldots, k_n$ is a {\em{maximal green sequence}} for the exchange matrix $B$ if it is a green sequence and all vertices $i \in \ex$ are red for $\mu_{k_n} \ldots \mu_{k_1} (B)$. 
\item[(d)] A sequence of mutations $k_1, \ldots, k_n$ is a {\em{reddening sequence}} (or a {\em{green-to-red sequence}}) for the exchange matrix $B$ if all vertices $i \in \ex$ are red for $\mu_{k_n} \ldots \mu_{k_1} (B)$. 
\end{enumerate}
\ede
%%%%%%%%%%%%%%%%%%%%%%%%%%%%%%%
\sectionnew{Quantum and Poisson CGL extensions}
\lb{prelim-CGL}
In this section we review the basics of quantum and Poisson CGL extensions, the structure of prime elements in chains of subalgebras, 
and the construction of quantum and classical cluster algebra structures on them.  
\subsection{Quantum CGL extensions} Let $\KK$ be an arbitrary base field, $A$ a $\KK$-algebra, $\sigma$ an algebra automorphism of $A$ and $\delta$ a 
$\sigma$-derivation of $A$ (i.e., $\delta(ab) = \delta(a) b + \sigma(a) \delta(b)$, $\forall a, b \in A$). 
A skew polynomial extension of a $\KK$-algebra $A$ is an algebra of the form 
\[
A[x; \sig, \de] := \oplus_{n\geq 0 } A x^n,
\]
where $A$ is a subalgebra and
\[
x a = \sig(a) x + \de(a) \quad {\mbox{for all}} \; \;  a \in A.
\]
Consider an iterated Ore extension (also called iterated skew polynomial algebra)
\begin{equation}
    \label{R}
R := \KK[x_1][x_2; \sigma_2, \delta_2] \cdots [x_N; \sigma_N, \delta_N], 
\end{equation}
and denote the intermediate algebras 
\[
R_k:=\KK[x_1][x_2; \sigma_2, \delta_2] \cdots [x_k; \sigma_k, \delta_k]
\]
for
$0 \leq k \leq N$. In particular, $R_0 = \KK$ and $R_N = R$. 

\bde{CGL} \hfill 
\begin{enumerate}
\item[(a)] An iterated Ore extension $R$ 
is called a (quantum) \emph{Cauchon--Goodearl--Letzter} (\emph{CGL}) \emph{extension} \cite{Ca,GL} if it is equipped with a rational action of a $\KK$-torus $T$ 
by $\KK$-algebra automorphisms satisfying the following conditions:
\begin{enumerate}
\item[(i)] The elements $x_1, \ldots, x_N$ are $T$-eigenvectors.
\item[(ii)] For every $2 \leq k \leq N$, $\delta_k$ is a locally nilpotent 
$\sigma_k$-derivation of $R_{k-1}$. 
\item[(iii)] For every $1 \leq k \leq N$, there exists $h_k \in T$ such that 
$\sigma_k = (h_k \cdot)$ and the $h_k$-eigenvalue of $x_k$, to be denoted by $\la_k$,  is not a root of unity. 
\end{enumerate}
\item[(b)] A CGL extension is called  {\em{symmetric}} if it is also a CGL extension when the 
generators $x_1, \ldots, x_N$ are adjoined in the reverse order
$$
R= \KK [x_N] [x_{N-1}; \sigma^*_{N-1}, \delta^*_{N-1}] \ldots [x_1; \sigma^*_1, \delta^*_1].
$$
The corresponding lambda scalars will be denoted by $\lambda_k^*$.
\end{enumerate}
\ede

This axiomatics comes from the works of Cauchon \cite{Ca} and Goodearl--Letzter \cite{GL}
who initiated the study of the ring theoretic properties of the algebras in this class.

For $1 \leq j < k \leq N$, let $\la_{kj} \in \KK^\times$ denote the $h_k$-eigenvalue of $x_j$, i.e., $h_k \cdot x_j = \la_{kj} x_j$.
The symmetricity of a CGL extension is equivalent to imposing a mild condition on the action of $T$ and the following abstract Levendorskii--Soibelman type straightening law (see \cite[Definition 3.12]{GY-Memo}):
\medskip

{\em{For all $j <k$, the element $x_k x_j - \la_{kj} x_j x_k$
belongs to the unital subalgebra of $R$ generated by $x_{j+1}, \ldots, x_{k-1}$ for some 
scalar $\la_{kj} \in \KK^\times$.}}

\bre{symmetric} The term symmetric in \deref{CGL}(2) does not refer to algebras that come from symmetric Kac--Moody algebras or quantum cluster algebras associated to quivers rather than valued quivers; all of those naturally appear in this setting. The term is referring to the condition that $R$ is a CGL extension in both the direct and reverse orders of adjoining its generators. 
\ere

If an algebra $A$ is equipped with a rational action of a $\KK$-torus $T$ by algebra automorphisms, then $A$ has a canonical grading by the rational character lattice $X(T)$ of $T$. An element $a \in A$ will be called homogeneous if it is homogeneous with respect to this grading.
Its weight will be denoted by
\begin{equation}
\label{wt}
\deg (a) \in X(T).
\end{equation}

Following Chatters \cite{Cha}, a nonzero element $p$ of a domain $A$ (a noncommutative ring without zero divisors) is called {\em{prime}} if it is {\em{normal}} (meaning that $Ap=pA$)
and the factor $A/pA$ is a domain. A noetherian domain $A$ is called a {\em{unique factorization domain}} (UFD) \cite{Cha} if every nonzero prime ideal of $A$ contains a 
prime element. It is easy to show that, in the commutative noetherian situation, this definition 
is equivalent to the more common one. If $A$ is a $\KK$-algebra equipped with a rational
action of a $\KK$-torus $T$ by algebra automorphisms, then $A$ is called a $T$-UFD \cite{LLR} if every nonzero $T$-prime ideal of $A$ contains a homogeneous prime element. Recall that a $T$-prime ideal of a (noncommutative) algebra $A$ is a $T$-stable ideal such that
\begin{equation}
IJ \subseteq P \Rightarrow I \subseteq P 
\; \; \mbox{or} \; \; J \subseteq P
\label{IJ}
\end{equation}
for all $T$-stable ideals $I$ and $J$ of $A$. 
%Every homogeneous normal elements of such an algebra has a unique up to associates factorzation into homogeneous prime elements. 
Launois, Lenagan and Rigal proved \cite{LLR} that every CGL extension is a $T$-UFD. 

For a function $\eta : [1,N] \to \Zset$ consider the {\em{predecessor}} and {\em{successor functions}} 
\[
[1,N] \to [1,N] \sqcup \{-\infty\} 
\quad 
\mbox{and} 
\quad
[1,N] \to [1,N] \sqcup \{\infty\}
\]
for its level sets defined by 
\[
k[-1] := 
\begin{cases}
\max \{ j <k \mid \eta(j) = \eta(k) \}, 
&\mbox{if $\exists j < k$ such that $\eta(j) = \eta(k)$} 
\\
- \infty, \; & \mbox{otherwise} 
\end{cases}
\]
and
\[
k[1]:=
\begin{cases}
\min \{ j > k \mid \eta(j) = \eta(k) \}, 
&\mbox{if $\exists j > k$ such that $\eta(j) = \eta(k)$} 
\\
\infty, & \mbox{otherwise}. 
\end{cases}
\]
Set $k[0]:=k$ and $k[\pm (m+1)] := k[\pm m][\pm 1]$ whenever the latter is defined. Define the {\em{order function}}
\begin{equation}
\label{order}
O_\pm : [1,N] \to \Nset, 
\quad 
O_\pm(k) := \max \{ m \in \Nset | k[\pm m ] \neq \pm \infty\}. 
\end{equation}

\bth{CGLprep}
{\rm \cite[Theorem 4.3]{GYprep}} The following hold for every CGL extension $R$ as in \eqref{R}:
\begin{enumerate}
\item[(a)] For each $1 \leq k \leq N$, the subalgebra $R_k$ has a unique up to scalar homogeneous prime element $p_k$ that does not belong to $R_{k-1}$. 
\item[(b)] There exist a function $\eta : [1,N] \to \Zset$ and elements
$$
c_k \in R_{k-1} \; \; \mbox{for all} \; \; k \in [2,N] \; \; 
\mbox{with} \; \; \delta_k \neq 0
$$
such that the homogeneous prime elements $p_1, \ldots, p_N$ (up to rescaling) are recursively given by
\[
p_k := 
\begin{cases}
p_{k[-1]} x_k - c_k, &\mbox{if} \; \;  \delta_k \neq 0 \\
x_k, & \mbox{if} \; \; \delta_k = 0. 
\end{cases}
\]
\item[(c)]
The function $\eta$ with these properties is not unique but its level sets, and as a consequence, the successor and predecessor functions are unique. Furthermore, $k[-1] = - \infty$ if and only if $\delta_k=0$.
\end{enumerate}
\eth
We note that the elements $p_1, \ldots, p_N$ are not prime elements of $R$, but all homogeneous prime elements of $R$ up to rescaling are 
\[
\{ p_k \mid k[1]= \infty\},
\]
and they are the frozen variables of a quantum cluster algebra structure on $R$ constructed in \cite{GY0,GY-Memo} and recalled below.

We will need the following two conditions that are satisfied for all interesting symmetric CGL extensions that we are aware of:
\medskip

{\bf{Condition (A).}} The base field $\KK$ contains square roots $\nu_{kj}$ of the scalars
$\la_{kj}$ for $1 \leq j < k \leq N$ such that the subgroup of $\KK^\times$ generated 
by all of them contains no elements of order $2$.
\medskip

{\bf{Condition (B).}} There exist positive integers $d_n$ for $n \in \range(\eta)$,
such that 
$$
\la_k^{d_{\eta(j)}} = \la_j^{d_{\eta(k)}} 
$$
for all $k, j \in [1,N]$, $p(k) \neq - \infty$, $p(j) \neq - \infty$.
\medskip

Denote the matrix 
\begin{equation}
\label{nu}
\boldsymbol{\nu} := (\nu_{kj})_{k,j=1}^N \in M_{N,N}(\KK^\times), 
\end{equation}
where $\nu_{kj}$ are given from Condition (A) for $j <k$, $\nu_{kk}=1$ and $\nu_{kj}:= \nu_{jk}^{-1}$ for $j> k$. 

To each permutation $\tau \in S_N$, we associate another permutation 
\begin{equation}
\label{taub}
\tau_\bullet \in \prod_{a \in \Zset} S_{\eta^{-1}(a)}
\end{equation}
defined as follows.
For $a$ in the range of $\eta$, set $|a|:= |\eta^{-1}(a)|$. Consider the 
set
\[
\eta^{-1}(a)= \{ \tau(k_1), \ldots, \tau(k_{|a|}) \mid k_1 < \cdots  < k_{|a|} \}
\]
and order its elements in an increasing order. Define $\tau_\bullet \in S_N$ by setting
$\tau_\bullet(\tau(k_i))$ to be equal to the $i$-th element in the list 
(for all choices of $a$ and $i$). 

For example, consider the case when $N=6$ and $\eta : [1,6] \to \Zset$ is given by 
$$
\eta(1)=\eta(4) = \eta(6)=1, \quad
\eta(2)=\eta(5)= 2, \quad
\eta(3) =3.
$$
Let 
$$
\tau := [3 \, 4 \, 2 \, 5 \, 1\, 6] \, \in S_6
$$ 
in the one-line notation for permutations. Then
$$
\eta^{-1}(1) = \{ \tau(5)=1 < \tau(2) = 4 < \tau(6) = 6 \}
$$
and thus, 
$$
\tau_\bu(4) =1, \quad \tau_\bu(1) = 4, \quad \tau_\bu(6) =6.
$$
Similarly, one finds that $\tau_\bu(j)=j$ for $j =2$, $3$ and $5$.

Denote the following subset of the symmetric group $S_N$:
\begin{multline}
\label{tau}
\Xi_N := \{ \tau \in S_N \mid 
\tau(k) = \max \, \tau( [1,k-1]) +1 \; \;
\mbox{or} 
\\
\tau(k) = \min \, \tau( [1,k-1]) - 1, 
\; \; \forall k \in [2,N] \}.
\end{multline}
Equivalently, $\Xi_N$ can be defined as 
\[
\Xi_N = \{ \tau \in S_N \mid \tau([1,k]) \; 
\mbox{is an interval for all} \; 2 \leq k \leq N \}.
\]
The identity element of $S_N$ will be denoted by $\id$. It belongs to $\Xi_N$.

Set
\begin{equation}
\label{frex}
\fr := \{ 1 \leq k \leq N \mid k[1] = + \infty \}
\quad
\mbox{and}
\quad
\ex := [1,N] \setminus \fr.
\end{equation}

For the square matrix ${\boldsymbol{\nu}}$ from Condition (A), see \eqref{nu},
denote by $\Om : \Zset^N \times \Zset^N \to \KK^\times$ the multiplicative bicharacter 
such that
\[
\Omega(e_l, e_j)=
\prod_{m=0}^{O_-(l)} \prod_{n=0}^{O_-(j)}
\mu_{l[-m], j[-n]}, \quad \forall 1 \leq l, j \leq N,
\]
where $e_1, \ldots, e_N$ are the standard generators of $\Zset^N$. 

\bth{cluster} \cite[Theorem 8.2 and Proposition 8.13]{GY-Memo} Let $R$ be a symmetric CGL extension satisfying conditions (A) and (B). Then $R$ is isomorphic to a quantum cluster algebra with frozen and exchangeable indices $\fr$ and $\ex$ given by \eqref{frex} for the function $\eta : [1,N] \to \Zset$ from \thref{CGLprep}
that has a subset of (explicitly constructed) seeds $\Sigma_\tau$ for $\tau \in \Xi_N$ such that the following hold:
\begin{enumerate}
\item[(a)] For every $l \in \ex$, there exists a unique vector $b^j=(b^l_i) \in \Zset^N$ such that 
\[
\sum_{i=1}^N b^l_i \deg(x_i)=0,
\]
recall \eqref{wt}, and 
\[
\Omega(b^l, e_j)=1, \quad \forall 1 \leq j \leq N, j \neq l, \quad \quad
\Omega(b^l, e_l)^2 = (\lambda_l^*)^{-1}.
\]
The exchange matrix of the initial seed $\Sigma:= \Sigma_\id$ is the matrix $\B$ with columns $b^j$ for $j \in \ex$. The cluster variables of $\Sigma$ are rescalings of the elements $p_1, \ldots, p_N$ from \thref{CGLprep}.
\item[(b)] For all $\tau, \tau' \in \Xi_N$ satisfying
$$
\tau' = ( \tau(k) \, \tau(k+1)) \tau = \tau (k \, (k+1))
$$
for some $k \in [1,N-1]$, the following hold:
\begin{enumerate}
\item[(i)]
If $\eta(\tau(k)) \neq \eta (\tau(k+1))$, then $\Sigma_{\tau'} = \Sigma_\tau$.
\item[(ii)] If $\eta(\tau(k)) = \eta (\tau(k+1))$, 
then $\Sigma_{\tau'} = \mu_{k_\bullet}(\Sigma_\tau)$, where $k_\bullet = \tau_\bullet \tau(k)$.
\end{enumerate}
\item[(c)] In the case (ii) in part (b), the only negative entries of the $k_\bullet$-th row of the exchange matrix of the seed $\Sigma_\tau$ are in rows $k_\bullet[-1]$ and $k_\bullet[1]$ and the positive entries are in rows $j$ such that $\eta(j) \neq \eta(k)$.
\end{enumerate}
The frozen variables of the quantum cluster algebra in question are not inverted.
\eth 

The quantum cluster algebras in the theorem are the multiparameter ones as developed in \cite[Chapter 2]{GY-Memo}. The CGL extensions that are quantum cluster algebras in the original sense of Berenstein and Zelevinsky \cite{BerZe} are the ones for which the scalars $\la_k$, $\la_{kj}$ are powers of an indeterminate $q$. 

Parts (a) and (b) of \thref{cluster} are from Theorem 8.2 in \cite{GY-Memo}. Part (c) of \thref{cluster} is the combination of Theorem 8.2(a) and Proposition 8.13(b) in \cite{GY-Memo}. 

\bre{negative} In Theorem \ref{tcluster} we consider the negative of the exchange matrices used in \cite{GY-Memo}. Of course, this does not change the cluster structure. We chose the sign of exchange matrices in order to get maximal green sequences from $\B$ to $\B_{w_\circ}$ instead of the other way around, where $w_\ci$ is the longest element of $S_N$. More concretely, the choice of sign ensures that at each step the mutation quiver has incoming arrows to the mutation vertex $k$ from $k[-1]$ and $k[1]$ instead of outgoing arrows. 
\ere

The role of $T$-UFDs is that in such algebras every homogeneous normal element is uniquely represented as a product of homogeneous prime elements (up to associates)
which is used to force the mutations in \cite[Theorem 8.2]{GY-Memo}; the entries of the exchange matrices record the exponents of the primes for certain 
homogeneous normal elements.  
\subsection{Poisson CGL extensions}
Let $A$ be a Poisson algebra over a base field $\KK$. A Poisson Ore extension is a Poisson algebra $A[x; \sig, \de]_p$, containing $A$ as a Poisson subalgebra, which is isomorphic to $A[x]$ as commutative algebra and whose Poisson bracket satisfies
\[
\{x, a\} = \sig(a) x + \de(a), \; \; \forall a \in A.
\]
Here $\sigma$ is a Poisson derivation of $A$ and $\delta$ is a skew derivation, meaning that it is a derivation of the commutative algebra B and
\[
\de ( \{a, b \}) = \{\de(a), b\} + \{a, \de(b)\} + \sig(a) \de(b) - \de(a) \sig(b), \quad \forall a,b \in A.
\]
For an iterated  Poisson-Ore extension
\begin{equation}
\label{PR}
R := \KK[x_1]_p [x_2; \sig_2, \delta_2]_p \cdots [x_N; \sig_N, \delta_N]_p
\end{equation}
and $0 \leq k \leq N$, denote
$$
R_k := \KK [ x_1, \dots, x_k ] = \KK[x_1]_p [x_2; \sig_2, \delta_2]_p \cdots [x_k; \sig_k, \delta_k]_p.
$$

\bde{PCGL}
\hfill 
\begin{enumerate}
\item[(a)]
An iterated Poisson-Ore extension $R$ as above is called a {\em{Poisson CGL extension}} 
 if it is equipped with a rational Poisson action of a torus $T$ such that
\begin{enumerate}
\item[(i)] The elements $x_1, \ldots, x_N$ are $T$-eigenvectors.
\item[(ii)] For every $k \in [2,N]$, the skew derivation $\de_k$ of $R_{k-1}$ is locally nilpotent. 
\item[(iii)] For every $k \in [1,N]$, there exists $h_k \in \Lie T$ such that 
$\sig_k = (h_k \cdot)|_{R_{k-1}}$ and the $h_k$-eigenvalue of $x_k$, to be denoted by $\la_k$, is nonzero.
\end{enumerate}
\item[(b)] A Poisson CGL extension is called  {\em{symmetric}} if it is also a Poisson CGL extension when the 
generators $x_1, \ldots, x_N$ are adjoined in the reverse order
$$
R= \KK [x_N] [x_{N-1}; \sigma^*_{N-1}, \delta^*_{N-1}]_p \ldots [x_1; \sigma^*_1, \delta^*_1]_p.
$$
The corresponding $\lambda$-scalars will be denoted by $\lambda_k^*$.
\end{enumerate}
\ede

The symmetricity property is equivalent to requiring 
\[
\delta_k(x_j) \in \KK[x_{j+1}, \ldots, x_{k-1}], \quad \forall 1 \leq j < k \leq N
\]
and a mild condition on the $T$-action, see
\cite[Definition 6.1]{GY-memo2}. 

For a CGL extension $R$, denote the $h_k$-eigenvalues $\lambda_{kj} \in \KK$ given by $h_k . x_j = \lambda_{kj} x_j$ for $1 \leq  j < k \leq N$. 
Set $\lambda_{kk}=0$ and $\lambda_{jk} = - \lambda_{kj}$ for $1 \leq k < j \leq N$. Denote the skew symmetric matrix
\begin{equation}
\label{la}
\boldsymbol{\la} := (\lambda_{kj})_{k,j=1}^N \in M_{N,N}(\KK).
\end{equation}  

A Poisson domain is a Poisson algebra $R$ which is an integral domain as a commutative algebra. A Poisson-prime element of a Poisson domain $R$ is a prime element $p \in R$ such that
\[
p \mid \{ p, r \}, \quad \forall r \in R. 
\]
These conditions are equivalent to saying that the ideal $(p)$ is a prime ideal and a Poisson ideal.
In the case when the base field $\KK$ is uncountable and algebraically closed, $p \in R$ is Poisson-prime if and only if $p$ is a prime element of $R$ and the zero locus $V(p) \subseteq \MaxSpec R$ is a union of symplectic cores in the sense of \cite{BG}. In the case when $\KK = \Cset$, one can replace ``symplectic cores'' with ``symplectic leaves'' in this statement, see \cite[Sect. 4.1]{GY-memo2}.

A Poisson-prime ideal of a Poisson algebra $R$ is a Poisson ideal $P$ such that \eqref{IJ} holds for all Poisson ideals $I$ and $J$ of $R$. A Poisson domain $R$ is a Poisson unique factorization domain (Poisson-UFD), if each nonzero Poisson-prime ideal of $R$ contains a Poisson-prime element. 

In the equivariant version of this concept, one considers a Poisson domain $R$ equipped with a Poisson action of a torus $T$. A $T$-Poisson-prime ideal of a Poisson algebra $R$ is a $T$-stable Poisson ideal $P$ such that \eqref{IJ} holds for all $T$-stable Poisson ideal $I$ and $J$ of $R$. A Poisson domain $R$ is a $T$-Poisson unique factorization domain (Poisson-UFD), if each nonzero $T$-Poisson-prime ideal of $R$ contains a homogeneous (with respect to the $X(T)$-grading) Poisson-prime element. All Poisson CGL extensions are $T$-Poisson UFDs \cite[Theorem 4.7]{GY-memo2}.

\bth{CGLprepP}
{\rm \cite[Theorem 5.5]{GY-memo2}} The following hold for every Poisson CGL extension $R$ of length $N$:
\begin{enumerate}
\item[(a)] For each $1 \leq k \leq N$, the subalgebra $R_k$ has a unique up to scalar homogeneous Poisson-prime element $y_k$ that does not belong to $R_{k-1}$. 
\item[(b)] There exist a function $\eta : [1,N] \to \Zset$ and elements
$$
c_k \in R_{k-1} \; \; \mbox{for all} \; \; k \in [2,N] \; \; 
\mbox{with} \; \; \delta_k \neq 0
$$
such that the homogeneous Poisson-prime elements $p_1, \ldots, p_N$ (up to rescaling) are recursively given by 
\[
p_k := 
\begin{cases}
p_{k[-1]} x_k - c_k, &\mbox{if} \; \;  \delta_k \neq 0 \\
x_k, & \mbox{if} \; \; \delta_k = 0. 
\end{cases}
\]
\item[(c)]
The function $\eta$ with these properties is not unique but its level sets, and as a consequence, the successor and predecessor functions are unique. Furthermore, $k[-1] = - \infty$ if and only if $\delta_k=0$.
\end{enumerate}
\eth

For the square matrix ${\boldsymbol{\la}}$ from \eqref{la},
denote by $\Theta : \Zset^N \times \Zset^N \to \KK^\times$ the additive bicharacter defined by 
\[
\Theta(e_l, e_j)=
\sum_{m=0}^{O_-(l)} \sum_{n=0}^{O_-(j)}
\mu_{l[-m], j[-n]}, \quad \forall 1 \leq l, j \leq N,
\]
recall \eqref{order}. 

\bth{cluster-P} \cite[Theorem 11.1 and Proposition 11.8]{GY-memo2} Let $R$ be a symmetric Poisson CGL extension such that
\begin{equation}
\label{int}
\lambda_k/\lambda_j \in \Qset_+, \quad \forall 1 \leq j \neq k \leq N.
\end{equation}
Then $R$ is isomorphic to a classical cluster algebra with frozen and exchangeable indices $\fr$ and $\ex$ given by \eqref{frex} for the function $\eta : [1,N] \to \Zset$ from \thref{CGLprepP}
that has a subset of (explicitly constructed) seeds $\Sigma_\tau$ for $\tau \in \Xi_N$ such that the following hold:
\begin{enumerate}
\item[(a)] For every $l \in \ex$, there exists a unique element $b^j=(b^l_i) \in \Zset^N$ such that 
\[
\sum_{i=1}^N b^l_i \deg(x_i)=0,
\]
recall \eqref{wt}, and 
\[
\Theta(b^l, e_j) = - \delta_{j,l} \lambda_l^*, \quad \forall 1 \leq j  \leq N.
\]
The exchange matrix of the initial seed $\Sigma:= \Sigma_\id$ is the matrix $\B$ with columns $b^j$ for $j \in \ex$. The cluster variables of $\Sigma$ are rescalings of the elements $p_1, \ldots, p_N$ from \thref{CGLprepP}.
\item[(b)] For all $\tau, \tau' \in \Xi_N$ satisfying
$$
\tau' = ( \tau(k) \, \tau(k+1)) \tau = \tau (k \,  (k+1))
$$
for some $k \in [1,N-1]$, the following hold:
\begin{enumerate}
\item[(i)]
If $\eta(\tau(k)) \neq \eta (\tau(k+1))$, then $\Sigma_{\tau'} = \Sigma_\tau$.
\item[(ii)] If $\eta(\tau(k)) = \eta (\tau(k+1))$, 
then $\Sigma_{\tau'} = \mu_{k_\bullet}(\Sigma_\tau)$, where $k_\bullet = \tau_\bullet \tau(k)$.
\end{enumerate}
\item[(c)] In the case (ii) in part (b), the only negative entries of the $k_\bullet$-th row of the exchange matrix of the seed $\Sigma_\tau$ are in rows $k_\bullet[-1]$ and $k_\bullet[1]$ and the positive entries are in rows $j$ such that $\eta(j) \neq \eta(k)$.
\end{enumerate}
The frozen variables of the quantum cluster algebra in question are not inverted.
\eth 
Parts (a) and (b) of \thref{cluster-P} are from Theorem 11.1 in \cite{GY-memo2} and part (c) of the theorem follows from Theorem 11.1(a) and Proposition 11.8(b) in \cite{GY-Memo}. For the reasons described in \reref{negative}, 
in \thref{cluster-P} we are using the negative of the exchange matrices used in \cite[Theorem 11.1]{GY-Memo}, which does not change the cluster structure. \bre{quant} In \cite{Mi-th}, Mi constructed a quantization of each symmetric Poisson CGL extension that is integral in the sense that the elements $h_k$
in \deref{PCGL} belong to the cocharacter lattice of $T$ and proved that the quantization is a symmetric (quantum) CGL
extension. This integrality assumption implies the weaker integrality condition \eqref{int} used in \thref{cluster-P}. Under the former assumption, one can show that the sets of exchange matrices and the 
mutations between them in Theorems \ref{tcluster} and \ref{tcluster-P} match with each other. In this case, our main theorems on maximal green sequences 
for quantum and Poisson CGL extensions can be deduced from each other. However, it is more natural to consider the proofs in parallel as instances of the notion of 
layered $T$-systems developed in the next section, instead of translating mutations of exchange matrices from one setting to the other. 
\ere

\sectionnew{Layered $T$-systems and maximal green sequences}
\label{Tsystres}
In this section we define a general notion of full layered $T$-systems and prove that they are maximal green sequences. 
\subsection{Layered $T$-systems}
Consider a sequence of mutations of exchange matrices without frozen variables
\begin{equation}
\label{mut-sequence}
B=B_1 \stackrel{k_1}{\longrightarrow} B_2 \stackrel{k_2}{\longrightarrow} \ldots \stackrel{k_r}{\longrightarrow} B_{r+1} \in M_{\ex, \ex}(\Zset)
\end{equation}
and assume that $\ex$ is totally ordered, i.e., 
\[
\ex \subseteq [1,N] 
\]
for some positive integer $N$.
\bde{Tsyst} The sequence of mutations \eqref{mut-sequence} will be called a {\em{layered $T$-system}} if there exists a function
\[
\eta : \ex \to \Zset
\]
satisfying the following conditions:
\begin{enumerate}
\item For all $1 \leq i \leq r$, the incoming arrows to the vertex $k_i$ of $Q(B_i)$ are simple arrows (without labels) 
from the vertices $k_i[-1]$ and $k_i[1]$. If one or both of these indices is $\pm \infty$, then there are no incoming arrows from it. 
\item The outgoing arrows from $k_i$ are only directed to the vertices $j \in \ex$ such that $\eta (j) \neq \eta(k_i)$. They are allowed 
to have arbitrary labels. 
\end{enumerate}
\ede 

A valued quiver whose set of vertices is indexed by a totally ordered set $\ex$, together with a function $\eta : \ex \to \Zset$, will 
be called a {\em{layered}} valued quiver. Quivers that appear in cluster algebras associated to Lie theory are layered 
since their cluster variables correspond to words in Weyl groups and the set of simple reflections (roots) provides the 
layering, see for instance \cite[Definitions 2.2 and 2.3]{BFZ05}.  

\bre{disp} Each layered valued quiver can be displayed on the plane by placing its vertices so that their $x$-coordinates appear in an increasing order and the $\eta$-function is given  
by the $y$-coordinate. The arrows of the quiver might intersect. 
%For a mutation sequence \eqref{mut-sequence} fix $1 \leq i \leq r$ and denote $k:= k_i$, $Q:=Q(B_i)$. 
\begin{figure}
\centering

\scalebox{0.7}{
\begin{tikzpicture}[
  every node/.style={font=\small},
  dot/.style={circle, draw, fill=white, inner sep=1.5pt},
  dashline/.style={blue, dashed, thick, dash pattern=on 6pt off 4pt},
  arr/.style={->, thick, shorten >=4pt, shorten <=4pt}
]

% Y levels
% t^(m) : y=4
% t^(1) : y=2.4
% t^(2) : y=1.6
% t     : y=0

% --- Dashed horizontal lines (80% longer: 6 -> 10.8) ---
\draw[dashline] (-0.5, 4)   -- (10.8, 4)   node[right] {$t^{(m)}$};
\draw[dashline] (-0.5, 2.4) -- (10.8, 2.4) node[right] {$t^{(1)}$};
\draw[dashline] (-0.5, 1.6) -- (10.8, 1.6) node[right] {$t^{(2)}$};
\draw[dashline] (-0.5, 0)   -- (10.8, 0)   node[right] {$t$};

% --- Points ---
\node[dot, label=above:$j^{(m)}$] (js)  at (9.8, 4)   {};
\node[dot, label=above:$j^{(1)}$] (j1)  at (1,   2.4) {};
\node[dot, label=above:$j^{(2)}$] (j2)  at (2.2, 1.6) {};

\node[dot, label={below:$k_i[-1]$}] (km1) at (2.0, 0) {};
\node[dot, label={below:$k_i$}]     (k)   at (2.8, 0) {};
\node[dot, label={below:$k_i[1]$}]  (kp1) at (3.6, 0) {};

% --- Arrows from kappa to j-vertices ---
\draw[arr] (k) -- (js) node[midway, right, xshift=2pt] {$(v_1^{(m)}, v_2^{(m)})$};
\draw[arr] (k) -- (j1) node[midway, left, xshift=(-8pt + 0.3cm)] {$(v_1^{(1)}, v_2^{(1)})$};
\draw[arr] (k) -- (j2) node[midway, right, xshift=(2pt - 0.22cm), yshift=8pt] {$(v_1^{(2)}, v_2^{(2)})$};

% --- Arrows from k[-1] and k[1] to k ---
\draw[arr] (km1) -- (k);
\draw[arr] (kp1) -- (k);

% --- Dots suggesting more lines ---
\node at (3.0, 3.0) {$\cdots$};

\end{tikzpicture}
}
\caption{Layered $T$-systems}
\label{fig:layerT}
\end{figure}
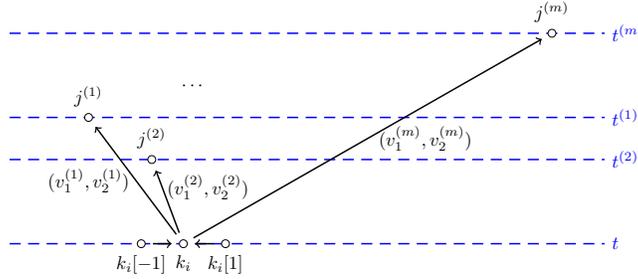
\deref{Tsyst} requires that for all $1 \leq i \leq r$, the arrows of the quiver $Q(B_i)$ incident to $k_i$ 
look like on Figure \ref{fig:layerT}: there are (at most two) simple incoming arrows from $k_i[-1]$ and/or $k_i[1]$ (for those indices that are finite) and those arrows are on the same row, and there are 
several outgoing (valued) arrows to vertices on other rows, labeled $j^{(1)}, \ldots, j^{(m)}$ on the figure.
On the right side of each row we display the value of $\eta$, i.e., the $y$-coordinate. The dotted lines represent the level sets of $\eta$; 
they are not a part of the quiver. 
\ere

For $t \in \range(\eta)$, denote 
\begin{equation}
\label{eta^-1}
\ell(t) := | \eta^{-1}(t) | \quad \mbox{and set} \quad  
\eta^{-1}(t) = \{ j_{t,1} < j_{t,2} < \ldots < j_{t, \ell(t)} \}. 
\end{equation}
\bde{Tsyst2} A layered $T$-system will be called {\em{full}} if for every $t \in \range(\eta)$, the mutations at the vertices of $\eta^{-1}(t)$ are shuffles of the 
$\ell(t)$ lists 
\begin{align*}
&j_{t,1}, j_{t,2}, \ldots, j_{t, \ell(t)-1}, j_{t, \ell(t)}, 
\\
&j_{t,1}, j_{t,2}, \ldots, j_{t, \ell(t)-1},
\\
&\ldots 
\\
&j_{t,1}, j_{t,2}, 
\\
&j_{t,1}.
\end{align*}
satisfying the condition that the index $j_{t,i}$ on the $n$-th row should come after the index $j_{t,i + 1}$ on the $(n-1)$-st row. 
\ede
\bex{1234} For example, if $\ex=[1,4]$ and $\eta$ is a constant function, then the sequence of mutations should have length 10 and the individual rows should be shuffled according to the diagram  
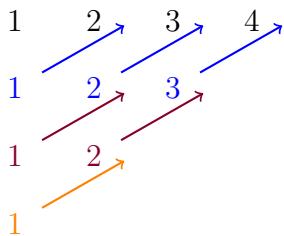
\begin{center}

\begin{tikzpicture}[
  every node/.style={font=\large},
  arr/.style={->, thick, shorten >=4pt, shorten <=4pt}
]

% Column spacing reduced by 25%: 1.4 * 0.75 = 1.05
% Direction vector scales accordingly: x-component 2.1 -> 1.575 (midpoint between col 2 and 3)
% i.e. midpoint of r1-2 and r1-3: x = (1.05 + 2.1)/2 = 1.575
% Arrow from r2-1 at (0,-0.9) to (1.575, 0): direction (1.575, 0.9)

% Row 1 (black): no arrows
\node (r1-1) at (0,    0) {1};
\node (r1-2) at (1.05, 0) {2};
\node (r1-3) at (2.1,  0) {3};
\node (r1-4) at (3.15, 0) {4};

% Row 2 (blue)
\node[blue] (r2-1) at (0,    -0.9) {1};
\node[blue] (r2-2) at (1.05, -0.9) {2};
\node[blue] (r2-3) at (2.1,  -0.9) {3};
\draw[arr, blue] (r2-1) -- (1.575,  0.0);
\draw[arr, blue] (r2-2) -- (2.625,  0.0);
\draw[arr, blue] (r2-3) -- (3.675,  0.0);

% Row 3 (dark red)
\node[purple!70!black] (r3-1) at (0,    -1.8) {1};
\node[purple!70!black] (r3-2) at (1.05, -1.8) {2};
\draw[arr, purple!70!black] (r3-1) -- (1.575, -0.9);
\draw[arr, purple!70!black] (r3-2) -- (2.625, -0.9);

% Row 4 (orange)
\node[orange] (r4-1) at (0, -2.7) {1};
\draw[arr, orange] (r4-1) -- (1.575, -1.8);

\end{tikzpicture}
\end{center}
where each arrow indicates that the integer on the $n$-th row should be placed after the integer on the $(n-1)$-st row as shown on the diagram. For example, the following is an allowed shuffle:
\[
1 \; 2 \; 3 \; \textcolor{blue}{1} \; 
\textcolor{blue}{2} \; 
\textcolor{purple!70!black}{1} \;
4 \; \textcolor{blue}{3} \; 
\textcolor{purple!70!black}{2} \; 
\textcolor{orange}{1}.
\]
\eex

\bre{T-syst} 
\hfill
\begin{enumerate}
\item[(a)] $T$-systems are discrete recurrence relations that often lead to sequences of mutations for cluster algebras coming from mathematical physics and representation theory 
with the property that, at every step, the mutation is done at an index such that  
there are 0, 1 or 2 incoming simple arrows to the corresponding vertex of the valued quiver, see e.g. \cite[Eq. (5)]{HL}.
The layered condition in \deref{Tsyst} 
is present in all such examples that we are aware of.  
\item[(b)]  Each full layered $T$-system has length 
\begin{equation}
\label{choose2}
\begin{pmatrix}
\eta +1 \\
2
\end{pmatrix} := \prod_{t \in \range{\eta}}
\begin{pmatrix}
    \ell(t) +1 \\
    2
\end{pmatrix} 
=
\prod_{t \in \range{\eta}}
\begin{pmatrix}
|\eta^{-1}(t)| +1 \\
2
\end{pmatrix} 
\end{equation}
because for each $t \in \range(\eta)$ the mutation sequence should have precisely 
\[
\ell(t)(\ell(t) + 1)/2 =  
\begin{pmatrix}
\ell(t) + 1 \\
2
\end{pmatrix}
\]
mutations at vertices in $\eta^{-1}(t)$. 
\end{enumerate}
\ere
\bth{T-syst-green} Every full layered $T$-system is a maximal green sequence. 
\eth
\bex{GLSseq} Gei\ss--Leclerc--Schr\"oer described a sequence of mutations of certain seeds of the cluster algebra structure on the coordinate ring of every unipotent cell of a symmetric Kac--Moody algebra \cite[pp. 399-400]{GLS1}. Keller conjectured that those mutation sequences are maximal green sequences \cite[Sect. 4.1]{K3}. This is a direct consequence of \thref{T-syst-green} because those mutation sequences are full layered $T$-systems. They are layered $T$-systems by \cite[Theorem 13.1]{GLS1} and are full because of Steps 
$1, 2, \ldots k$ on pp. 399--400 of \cite{GLS1}. Their lengths are computed on the first display on p. 401 in \cite{GLS1} and are of the form \eqref{choose2} because the level sets of the corresponding $\eta$ function have $t_1, t_2, \ldots t_n$ elements in the notation of \cite{GLS1}. 
\eex
\subsection{Mutations of the $A_n$ quiver} For the proof of \thref{T-syst-green} we will need several simple facts about mutations of $A_n$ quivers. 
Denote by $\overleftarrow{A}_n$ the quiver $1 \leftarrow 2 \leftarrow \ldots \leftarrow n$. 
\ble{Anmutations} Let 
\[
Q_1 := \overleftarrow{A}_n \stackrel{k_1}{\longrightarrow} Q_2 \stackrel{k_2}{\longrightarrow} \ldots \stackrel{k_r}{\longrightarrow} Q_{r+1}
\]
be a sequence of mutations which is a shuffle of the lists 
\begin{align*}
&1, 2, \ldots, n-1, n
\\
&1, 2, \ldots, n-1, 
\\
&\ldots 
\\
& 1,  2,
\\
& 1
\end{align*}
according to the rule of \deref{Tsyst2}. In particular, $r = \begin{pmatrix} n+1 \\ 2 \end{pmatrix}$. Then the following hold:
\begin{enumerate}
\item[(a)] Each of the quivers $Q_i$ is the $A_n$ graph with some orientation of its edges and $k_i$ is a sink of $Q_i$. 
\item[(b)]  The sequence of mutations is a maximal green sequence. 
\end{enumerate}
\ele 
\begin{proof} (a) We prove the statement by induction on $i$, the case $i=1$ being obvious. Assume its validity for the steps before the $i$-th one. Since, 
by the induction hypothesis, for $0 \leq p \leq i -1$, $k_p$ is a sink 
of $Q_p$, the quiver $Q_{p+1}$ is obtained from $Q_p$ by reversing the direction of the arrows adjacent to $k_p$ for all $1 \leq p \leq  i -1$. 
The induction hypothesis for $Q_{i -1}$ implies that $Q_i$ is the $A_n$ quiver with some orientation of its edges. In the passage  
$Q_1 \stackrel{k_1}{\longrightarrow} Q_2 \stackrel{k_2}{\longrightarrow} \ldots \stackrel{k_{i -1}}{\longrightarrow} Q_i$, 
the direction of the arrow $(k_i-1) \leftarrow k_i$  is reversed
\[
|\{k_p \mid 1 \leq p \leq i -1, k_p = k_i-1 \; \mbox{or} \; k_i\}| = 2 |\{k_p \mid 1 \leq p \leq i -1, k_p = k_i \}| +1 
\]
 times, so $Q_i$ has an arrow from $k_i -1$ to $k_i$. Analogously, in the passage from $Q_1$ to $Q_i$, the direction of the arrow $k_i \leftarrow (k_i +1)$ 
 is reversed
 \[
 |\{k_p \mid 1 \leq p \leq i -1, k_p = k_i \; \mbox{or} \; k_i +1 \}|= 2 |\{k_p \mid 1 \leq p \leq i -1, k_p = k_i\}|
 \]
 times, so $Q_i$ has an arrow from $k_i + 1$ to $k_i$. 

Part (b) can be proved in a similar manner. It also follows from the Garver--Musiker results \cite{GM} for maximal green sequences of $A_n$ quivers.  
\end{proof}
\subsection{Proof of \thref{T-syst-green}} 
 
 \bde{truncation} Consider a function $\eta : \ex \to \Zset$ and $t \in \range(\eta)$.
 \begin{enumerate}
 \item[(a)] Define the $t$-{\em{truncation}} of a mutation sequence $\boldsymbol{k} = (k_1, \ldots, k_i)$ 
to be the subsequence of indices that belong to $\eta^{-1}(t)$. It will be denoted by $\boldsymbol{k}\langle t \rangle$. 
\item[(b)] For a valued quiver $Q$ on the set $\ex$, define the $t$-truncation $Q\langle t \rangle$ to be the full valued subquiver with vertex set $\eta^{-1}(t)$. 
\item[(c)] For a valued ice quiver $Q$ on the set $\ex \sqcup \ex'$ with exchangeable vertices $\ex$, define the 
$t$-truncation $Q\langle t \rangle$ to be the full valued ice subquiver with vertex set $\eta^{-1}(t) \sqcup \big( \eta^{-1}(t) \big)'$. 
\end{enumerate}
\ede

For a totally ordered set $S$ of cardinality $n$ denote by 
\[
\overleftarrow{A}(S)
\]
the quiver $\overleftarrow{A}_n$ whose index set $[1,n]$ is replaced by $S$ in the same order.

\bth{ind-Tsyst} Consider a full layered $T$-system as in \eqref{mut-sequence}. For any $1 \leq i \leq r$, denote the mutation subsequence 
\[
\boldsymbol{k}_i := (k_1, \ldots, k_{i-1}).
\]
The following hold for the arrows of the quiver $\mu_{\boldsymbol{k}_i} \big( \, \widehat{Q(B)} \, \big)$:
\begin{enumerate}
\item[(a)] For all $s \neq t \in \range(\eta)$, there are no arrows between any pair of vertices in the sets $\eta^{-1}(s)$ and
$\big(\eta^{-1}(t)\big)'$.   

\item[(b)] For any $t \in \range(\eta)$, the $t$-truncated quiver 
\[
\Big( \mu_{\boldsymbol{k}_i} \big( \, \widehat{Q(B)} \, \big) \Big) \langle t \rangle 
\quad 
\mbox{is canonically isomorphic to} 
\quad \mu_{\boldsymbol{k}_i \langle t \rangle } \Big( \, \widehat{\overleftarrow{A} ( \eta^{-1}(t) )} \, \Big)
\]
via the canonical bijection between the index sets of their vertices. 
\end{enumerate}
\eth
\begin{proof} We first prove by induction on $i$, statement (a) and the statement 
\begin{enumerate}
\item[(b')] For any $t \in \range(\eta)$, the arrows between any pair of vertices in the sets $\eta^{-1}(t)$ and
$\big(\eta^{-1}(t)\big)'$, in the valued ice quivers
\[
\Big( \mu_{\boldsymbol{k}_i} \big( \, \widehat{Q(B)} \, \big) \Big) \langle t \rangle 
\quad 
\mbox{and} 
\quad \mu_{\boldsymbol{k}_i \langle t \rangle } \Big( \, \widehat{\overleftarrow{A} ( \eta^{-1}(t) )} \, \Big)
\]
are the same.
\end{enumerate}
The case $i=1$ is obvious. Denote for brevity 
\[
t_i := \eta(k_i).
\]

The definition of layered $T$-system, the inductive assumption and \leref{Anmutations}(b) 
imply that the arrows of the quiver $\mu_{\boldsymbol{k}_i} \big( \, \widehat{Q(B)} \, \big)$ that are incident to $k_i$ 
are as displayed on Figure \ref{fig:loc}: 
\begin{enumerate}
\item Incoming simple arrows from the vertices $k_i[-1]$ and/or  $k_i[1]$; if one or both of these indices is $\pm \infty$, 
then there are no incoming arrows from it.
\item Outgoing valued arrows to vertices with different $\eta$-value than $t_i$.  
\item Outgoing simple arrows to some vertices $j'_1, j'_2, \ldots, j'_u \in \big( \eta^{-1}(t_i) \big)^{-1}$ and these are precisely 
the arrows between $k_i$ and $\big( \eta^{-1}(t_i) \big)'$ in the quiver $\mu_{\boldsymbol{k}_i \langle t_i \rangle } \Big( \, \widehat{\overleftarrow{A} ( \eta^{-1}(t_i) )} \, \Big)$.
\end{enumerate}
In particular, $k_i$ is a green vertex of the quiver $\mu_{\boldsymbol{k}_i} \big( Q(B) \big)$. 
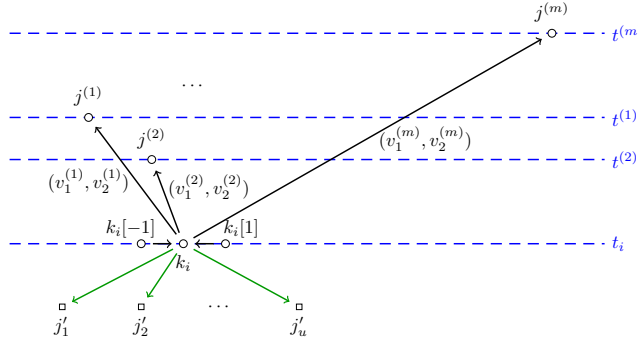
\begin{figure}
    \centering
    
\scalebox{0.7}{
\begin{tikzpicture}[
  every node/.style={font=\small},
  dot/.style={circle, draw, fill=white, inner sep=1.5pt},
  sq/.style={rectangle, draw, fill=white, inner sep=1.5pt},
  dashline/.style={blue, dashed, thick, dash pattern=on 6pt off 4pt},
  arr/.style={->, thick, shorten >=4pt, shorten <=4pt}
]

% Y levels
% t^(m) : y=4
% t^(1) : y=2.4
% t^(2) : y=1.6
% t     : y=0

% --- Dashed horizontal lines (80% longer: 6 -> 10.8) ---
\draw[dashline] (-0.5, 4)   -- (10.8, 4)   node[right] {$t^{(m)}$};
\draw[dashline] (-0.5, 2.4) -- (10.8, 2.4) node[right] {$t^{(1)}$};
\draw[dashline] (-0.5, 1.6) -- (10.8, 1.6) node[right] {$t^{(2)}$};
\draw[dashline] (-0.5, 0)   -- (10.8, 0)   node[right] {$t_i$};

% --- Points ---
\node[dot, label=above:$j^{(m)}$] (js)  at (9.8, 4)   {};
\node[dot, label=above:$j^{(1)}$] (j1)  at (1,   2.4) {};
\node[dot, label=above:$j^{(2)}$] (j2)  at (2.2, 1.6) {};

\node[dot, label={[xshift=-0.2cm, yshift=-0.1cm]above:$k_i[-1]$}] (km1) at (2.0, 0) {};
\node[dot, label={below:$k_i$}]     (k)   at (2.8, 0) {};
\node[dot, label={[xshift=0.3cm, yshift=-0.1cm]above:$k_i[1]$}]  (kp1) at (3.6, 0) {};

% --- Arrows from kappa to j-vertices ---
\draw[arr] (k) -- (js) node[midway, right, xshift=2pt] {$(v_1^{(m)}, v_2^{(m)})$};
\draw[arr] (k) -- (j1) node[midway, left, xshift=(-8pt + 0.3cm)] {$(v_1^{(1)}, v_2^{(1)})$};
\draw[arr] (k) -- (j2) node[midway, right, xshift=(2pt - 0.2cm), yshift=8pt] {$(v_1^{(2)}, v_2^{(2)})$};

% --- Arrows from k[-1] and k[1] to k ---
\draw[arr] (km1) -- (k);
\draw[arr] (kp1) -- (k);

% --- Dots suggesting more lines ---
\node at (3.0, 3.0) {$\cdots$};

% --- New vertices j'_1, ..., j'_u below all lines (y=-1.2) ---
\node[sq, label=below:$j'_1$] (jp1) at (0.5,  -1.2) {};
\node[sq, label=below:$j'_2$] (jp2) at (2.0,  -1.2) {};
\node at (3.5, -1.2) {$\cdots$};
\node[sq, label=below:$j'_u$] (jpu) at (5.0,  -1.2) {};

% --- Green arrows from k to j'-vertices ---
\draw[arr, green!60!black] (k) -- (jp1);
\draw[arr, green!60!black] (k) -- (jp2);
\draw[arr, green!60!black] (k) -- (jpu);

\end{tikzpicture}
}
\caption{Local picture at the mutation}
\label{fig:loc}
\end{figure}

Properties (1)--(3) imply that when performing the mutation
$\mu_{k_i}$ to $\mu_{\boldsymbol{k}_i} \big( \, \widehat{Q(B)} \, \big)$, arrows incident to $k_i$ should be reversed and arrows from $k_i[\pm 1]$ to 
$j^{(1)}, j^{(2)}, \ldots j^{(m)}$ and $j'_1, \ldots, j'_u$ should be added (the first set of arrows should be valued, the second not) and cancellation of pairs of opposite arrows should be performed. 

Therefore, the mutation 
\begin{equation}
\label{mutation-i}
\mu_{\boldsymbol{k}_i} \big( \, \widehat{Q(B)} \, \big) \stackrel{k_i}{\longrightarrow} \mu_{k_i} \mu_{\boldsymbol{k}_i} \big( \, \widehat{Q(B)} \, \big)
\end{equation}
has no effect on the arrows of the quiver 
\begin{enumerate}
\item[(i)] between the sets of vertices $\ex \backslash \eta^{-1}(t_i)$ and $\ex'$, and
\item[(ii)] within $\ex \backslash \eta^{-1}(t_i)$. 
\end{enumerate}

It follows from (i) that the validity of (a) for $i-1$ implies its validity for $i$. It follows from (ii) that the validity of (b') for $i-1$ and $t \neq t_i$ 
implies its validity for $i$ and $t \neq t_i$. 

We are left with proving the induction step for statement (b') for $t = t_i$. Properties (1)--(3) imply that the effect of the mutation \eqref{mutation-i} 
on the arrows between the sets of vertices $\eta^{-1}(t_i)$ and $\big(\eta^{-1}(t_i) \big)$ is the same as the effect of the mutation $\mu_{k_i}$ on the arrows 
of the quiver $\mu_{\boldsymbol{k}_i \langle t_i \rangle } \Big( \, \widehat{\overleftarrow{A} ( \eta^{-1}(t_i) )} \, \Big)$
between pairs of vertices in the same sets. Hence, the validity of (b') for $i-1$ and $t=t_i$ implies its validity for $i$ and $t=t_i$. 

Finally, properties (1)--(3) imply that the effect of the mutation \eqref{mutation-i} on arrows within $\eta^{-1}(t_i)$ is that all arrows incident to $t_i$ are reversed. 
This fact, properties (b') and (ii) and \leref{Anmutations}(a) imply part (b) of the theorem. 
\end{proof}
\noindent
{\em{Proof of \thref{T-syst-green}.}} The fact that $k_i$ is a green vertex of the quiver $\mu_{\boldsymbol{k}_i} \big( Q(B) \big)$
for all $1 \leq i \leq r$ was noted in the proof of \thref{ind-Tsyst}.

Each vertex of $\mu_{\boldsymbol{k}_r} \big( Q(B) \big)$ belongs to one of the $t$-truncations 
$\Big( \mu_{\boldsymbol{k}_r} \big( \, \widehat{Q(B)} \, \big) \Big) \langle t \rangle$. By \thref{ind-Tsyst}(b) the $t$-truncation is isomorphic to 
$\mu_{\boldsymbol{k}_i \langle t \rangle } \Big( \, \widehat{\overleftarrow{A} ( \eta^{-1}(t) )} \, \Big)$, which is the last member associated 
to the maximal green sequence in \leref{Anmutations}(b). 
Therefore, for every $t \in \range(\eta)$, all vertices of the $t$-truncation $\big( \mu_{\boldsymbol{k}_r} ( Q(B) ) \big) \langle t \rangle$
are red. Invoking \thref{ind-Tsyst}(a), we obtain that all vertices of  $\mu_{\boldsymbol{k}_r} \big( Q(B) \big)$ are red. (In the very last part, we could have used 
sign coherence for $c$-vectors, but the presented argument is more elementary and its components are needed for the rest of the proof.)
\qed
\bre{restr} \thref{ind-Tsyst} implies that every full layered $T$-system is a sequence of component preserving mutations in the sense of \cite{BMRYZ} and that at each step the first condition in \cite[Definition 3.6]{BMRYZ} is satisfied.
\ere
%%%%%%%%%%%%
\sectionnew{Maximal green sequences for quantum and Poisson CGL extensions}
In this section we prove the main theorem in the paper constructing maximal green sequences for the quantum and classical cluster algebra structures on symmetric quantum and Poisson CGL extensions.  
\label{resCGL}
\subsection{Paths in $\Xi_N$} 
Recall that $w_\ci$ denotes the longest element of $S_N$,
\[
w_\ci = [N \, (N-1) \, \ldots \, 1] \in S_N.
\]
Obviously, it is an element of $\Xi_N$.
Given $0 \leq k < N$ and a sequence 
\[
k < j_k \leq \cdots \leq j_1 \leq N,
\]
define
\[
\tau_{(j_k, \ldots, j_1)} 
:= (k \, (k+1)\,  \ldots\,  j_k) 
\ldots
(2 \, 3 \, \ldots \, j_2)
(1 \, 2 \, \ldots \, j_1) \in S_N,
\]
where in the right hand side we use the standard notation for cycles in $S_N$. In the one-line notation for elements of $S_N$, 
\begin{equation}
\label{tauj}
\tau_{(j_k, \ldots, j_1)} = 
[\ldots k \ldots (k-1) \ldots 1 \ldots],
\end{equation}
where the dots represent the list of integers $(k+1), (k+2), \ldots, n$ (in this order) and the integers $k, k-1, \ldots, 1$ are in positions $j_k -k +1, \ldots, j_2 -1, j_1$.
The elements \eqref{tauj} belong to $\Xi_N$ 
and every element of $\Xi_N$ is of this form \cite[Lemma 5.5]{GY-Memo}. For example, 
\[
w_0 = \tau_{(N, N,\ldots, N)}, \quad 
\mbox{where} \; k= N-1.
\]

\bre{another-tau}
A slightly different convention was used 
in \cite{GY-Memo} (see Section 5.2 there) where it was allowed that $j_k =k$ but it was required that $k \geq 0$. It is straightforward to check that the collections of $\tau$'s  there and in this paper are the same. In the convention of \cite{GY-Memo}, there is a non-uniqueness of representing the identity permutation $\id$ as a $\tau$; all other elements of $\Xi_N$ have a unique presentation as a $\tau$. In the present convention, every element of $\Xi_N$ has a unique presentation as a $\tau$, which follows at once from \eqref{tauj}. The only element coming from the case $k=0$ is the identity element 
\[
\id = \tau_{\varnothing}. 
\]
\ere
\bde{shuf} A {\em{contiguous path}} is a path in $S_N$ of the form 
\begin{equation}
\label{path}
\boldsymbol{\nu}:= 
\{ \id = \nu_1 \mt \nu_2 \mt \ldots \mt \nu_{M+1} = w_0 \}
\end{equation}
obtained by pulling 1 to the right, 2 to the right, etc., in the one-line notation for 
$\id =[1 \, 2 \, \ldots \, N]$
by following the rules
\begin{enumerate}
\item we can start pulling the integer $k$ to the right if $1, \ldots, k-1$ are already pulled to the right of it and 
\item at each step we interchange only two adjacent integers $m$ and $n$ such that $m < n$ and $m$ stays to the left of $n$. 
\end{enumerate}
\ede
\bre{shuff} In \leref{paths} we prove that every contiguous path is a path in $\Xi_N$. The term contiguous is chosen to reflect the fact that the elements of $\Xi_N$ have the property that they map each interval $[1,k]$ to a contiguous set of integers. 
\ere

It follows from condition (2) in \deref{shuf} that the length of every contiguous path is 
\[
M = 
\begin{pmatrix}
N \\ 2 
\end{pmatrix}.
\]

The simplest contiguous path is the path 
\begin{align}
\id = &[\circled{1} \, \circled{2} \, 3 \, \ldots \, N] \mt \ldots \mt
\nn
\\
&[\circled{2} \, 3 \, \ldots \, N \, \circled{1}] \mt \ldots \mt
\nn
\\
&[3 \, \ldots \, N \, \circled{2} \, \circled{1}] \mt \cdots \mt
\label{pathw0}
\\
&[(N-1) \, N \, \ldots \, 2 \, 1] \mt \ldots \mt
\nn
\\
&[N \, (N-1) \, \ldots \, 2 \, 1] = w_\ci
\nn
\end{align}
The circles indicate the integers that are being pulled to the right. 

\ble{paths} All contiguous paths are paths in $\Xi_N$. 
\ele
\begin{proof} We verify by induction on  $1 \leq i \leq M+1$ that, for every contiguous path as in \eqref{path}, $\nu_i \in \Xi_N$, the case $i=1$ being obvious. If $\nu_i \in \Xi_N$, then 
\[
\nu_i = \tau_{(j_k, \ldots, j_1)}
\]
for some $0 \leq k <N$ and 
$k < j_k \leq \ldots \leq j_1$. Conditions (1)--(2) in \deref{shuf} imply that either 
\[
\nu_{i+1} = \tau_{(j_k, \ldots, (j_s+1), \ldots, j_1)}
\]
and $j_{s-1} \geq  j_s + 1$ or 
\[
\nu_{i+1} = \tau_{(k+2, j_k, \ldots, j_1)}
\]
and $j_k \geq k + 2$, which proves the induction step.
\end{proof}

Consider a contiguous path $\boldsymbol{\nu}$ as in \eqref{path}. Condition (2) in \deref{shuf} implies that, for every step $1 \leq i \leq M$, there exist integers $1 \leq m_i < n_i \leq N$ and $1 \leq p_i \leq N$ 
such that 
\begin{align*}
\nu_i &= [ \ldots m_i \; n_i \ldots],
\\
\nu_{i+1} &= [\ldots n_i \; m_i \ldots],
\end{align*}
where $m_i$ and $n_i$ sit in positions $p_i$ and $p_i +1$ in $\nu_i$, and there are no differences in the other positions of $\nu_i$ and $\nu_{i+1}$. Therefore,
\begin{equation}
\label{nustep}
\nu_{i+1} = (m_i \, n_i) \nu_i = \nu_i s_{p_i}. 
\end{equation}
By iterating \eqref{nustep}, we obtain  
\[
\nu_i = s_{p_1} s_{p_2} \ldots s_{p_{i-1}}
\]
and, in particular,
\begin{equation}
\label{red-expr}
w_\ci = \nu_{M+1} = s_{p_1} s_{p_2} \ldots s_{p_M}.
\end{equation}
Since the length of $w_\ci$ is 
$M = 
\begin{pmatrix}
N \\ 2 
\end{pmatrix}
$, 
$\boldsymbol{w} := s_{p_1} s_{p_2} \ldots s_{p_M}$ is a reduced expression of $w_\ci$ and $\nu_i$ are the initial subwords in it:
\[
\nu_i = \boldsymbol{w}_{\leq i-1}, \quad \forall 1 \leq i \leq M+1. 
\]

Conversely, if $\boldsymbol{w} := s_{p_1} s_{p_2} \ldots s_{p_M}$ is a reduced expression of $w_\ci$ with the property that all of its initial subwords belong to $\Xi_N$, that is
\[
\boldsymbol{w}_{\leq i} \in \Xi_N, \quad 
\forall 1 \leq i \leq M-1,
\]
then one easily shows that
\begin{equation}
    \label{wred}
\id = \boldsymbol{w}_{\leq 0} \mt 
\boldsymbol{w}_{\leq 1} \mt \boldsymbol{w}_{\leq 2} \mt \ldots \mt \boldsymbol{w}_{\leq M} = w_\ci
\end{equation}
is a contiguous path. Indeed, since $\boldsymbol{w}$ is a reduced expression,
\[
\boldsymbol{w}_{\leq i-1}(p_i) <   
\boldsymbol{w}_{\leq i-1}(p_i +1).
\]
If $n_i < m_i$ denote the two integers, then $\boldsymbol{w}_{\leq i} = (n_i \, m_i) \boldsymbol{w}_{\leq i- 1}$ and $\boldsymbol{w}_{\leq i}$ is obtained from $\boldsymbol{w}_{\leq i- 1}$ by pulling $n_i$ to the right of $m_i$ in the one-line notation, 
showing that condition (2) in \deref{shuf} holds. Condition (1) in \deref{shuf} is verified recursively by using the assumption on the initial subwords. 

This gives us the following characterization of contiguous paths:

\bpr{shuf} There is a bijection between the set of contiguous paths (in $\Xi_N$) and the reduced expressions $\boldsymbol{w}$ of the longest element $w_\ci \in S_N$ such that each initial subword $\boldsymbol{w}_{\leq i}$ is in $\Xi_N$ for $1 \leq i \leq M-1$. The bijection is given by 
sending $\boldsymbol{w}$ to the path \eqref{wred}.
\epr
\subsection{Maximal green sequences from contiguous paths} Let $R$ be a quantum or Poisson CGL extension subject to the mild assumptions in Theorems \ref{tcluster} and \ref{tcluster-P}, respectively. By those two theorems, we have a collection of quantum (classical) seeds $\Sigma_\tau$ indexed by the elements $\tau$ of the subset $\Xi_N \subset S_N$. Denote by $B_\tau$ the principal part of the exchange matrix $\B_\tau$ of each of these seeds $\Sigma_\tau$, and set $B:= B_{\id}$ for the initial seed. 

Consider a contiguous path $\boldsymbol{\nu}$ as in \eqref{path}. By the results in the previous subsection, 
\[
\nu_i = \nu_{i-1} s_{p_i}, \quad \forall 1 \leq i \leq M = 
\begin{pmatrix}
N \\
2
\end{pmatrix}
\]
for a sequence $p_1, \ldots, p_M$ in [1,N] such that $\boldsymbol{w}:=s_{p_1} \ldots s_{p_M}$ is a reduced expression of $w_\ci$ (so, $\nu_i = \boldsymbol{w}_{\leq i-1}$). Recall the notation \eqref{taub}. Parts (b) of Theorems \ref{tcluster} and \ref{tcluster-P} imply that we have the following:
\begin{enumerate}
\item[(i)]
If $\eta(\nu_{i-1} (p_i)) \neq \eta (\nu_{i-1}(p_{i}+1))$, then $\Sigma_{\nu_i} = \Sigma_{\nu_{i-1}}$. In particular, 
\begin{equation}
\label{B=B}
\eta(\nu_{i-1} (p_i)) \neq \eta (\nu_{i-1}(p_{i}+1)) \quad \Rightarrow \quad B_{\nu_i} = B_{\nu_{i-1}}.
\end{equation}
\item[(ii)] If $\eta(\nu_{i-1}(p_i)) = \eta (\nu_{i-1}(p_i+1))$, 
then $\Sigma_{\nu_i} = \mu_{p_{i,\bullet}}(\Sigma_{\nu_{i-1}})$, where 
\[
p_{i,\bullet} = (\nu_i)_\bullet \nu_i(p_i).
\]
In particular,
\begin{equation}
    \label{BmuB}
   \eta(\nu_{i-1}(p_i)) = \eta (\nu_{i-1}(p_i+1))
   \quad \Rightarrow \quad 
   B_{\nu_i} = \mu_{p_{i,\bullet}}(B_{\nu_{i-1}})
\end{equation}
\end{enumerate}
Denote 
\[
\mathrm{seq}(\boldsymbol{\nu})_i
:= 
\begin{cases}
(\nu_i)_\bullet \nu_i(p_i), &\mbox{if} \; \; \eta(\nu_{i-1}(p_i)) = \eta (\nu_{i-1}(p_i+1))
\\
\varnothing, &\mbox{otherwise}. 
\end{cases}
\]
Finally, for each contiguous path $\boldsymbol{\nu}$, 
define the mutation sequence
\[
\mathrm{seq}(\boldsymbol{\nu}) =
\mathrm{seq}(\boldsymbol{\nu})_1 \, 
\mathrm{seq}(\boldsymbol{\nu})_2 \, 
\ldots \, 
\mathrm{seq}(\boldsymbol{\nu})_M.
\]

\bth{CGLgreen} Let $R$ be a quantum or Poisson CGL extension subject to the mild assumptions in Theorems \ref{tcluster} and \ref{tcluster-P}. Then, for every contiguous path $\boldsymbol{\nu}$, 
the mutation sequence $\mathrm{seq}(\boldsymbol{\nu})$ is a maximal green sequence
for the principal part $B$ of the exchange matrix of the initial seed $\Sigma$ of the quantum or classical cluster algebra structure on $R$ from 
Theorems \ref{tcluster} and \ref{tcluster-P}. The lengths of all such maximal green sequences equal 
\[
\begin{pmatrix}
\eta +1 \\
2
\end{pmatrix}
\]
in the notation \eqref{choose2}. 
\eth 

For example, the simplest contiguous path \eqref{pathw0} gives rise to a maximal green sequence.

\thref{CGLgreen} follows from \thref{T-syst-green}, \reref{T-syst}(b) and the following result:
\bpr{CGLtoTsyst} For every contiguous path $\boldsymbol{\nu}$ for a quantum or Poisson CGL extension satisfying the assumptions of Theorems \ref{tcluster} and \ref{tcluster-P}, the mutation sequence $\mathrm{seq}(\boldsymbol{\nu})$ is a full layered $T$-system for the function $\eta$ from Theorems \ref{tcluster} and \ref{tcluster-P}. 
\epr
\begin{proof} The mutation sequence $\mathrm{seq}(\boldsymbol{\nu})$ is a layered $T$-system due to  Theorems \ref{tcluster}(c) and \ref{tcluster-P}(c). The choice of sign of exchange matrices in those theorems, discussed in \reref{negative}, plays a role here. Otherwise we would have 
outgoing arrows from the mutation vertex $k$ to $k[-1]$ and $k[1]$ instead of incoming arrows. 

We are left with showing the fullness property of $\mathrm{seq}(\boldsymbol{\nu})$.  
Fix an integer $1 \leq k \leq N$. 
Consider the steps in the path $\boldsymbol{\nu}$ which move $k$ one position to the right in the one-line notation for permutations. Condition (2) of \deref{shuf} gives that before those steps occur, the integers $1, 2, \ldots, k-1$ were pulled to the right of $k$. By condition (1) of the definition, the steps we are considering are moving $k$ to the right of $k+1, k+2, \ldots, N$ in the same order (but there could be various other steps in between). The steps where $k$ is moved over $j>k$
and $\eta(k) \neq \eta(j)$ do not result in any mutations by Eq. \eqref{B=B}. Denote     
\[
\eta^{-1}(\eta(k)) = \{ j_1 < j_2 < \ldots < j_l \} \subseteq \{ k, k[1], \ldots, k[O^+(k)] \},
\]
recall \eqref{order}. Eq. \eqref{BmuB} implies that the steps when $k$ is consecutively moved to the right of $k[1], k[2], \ldots, k[O_+(k)]$, give the mutation sequence
\[
j_1 \, j_2 \, \ldots \, j_{O_+(k)}.
\]
(Note, that this sequence is empty if $O_+(k)=0$.)
In other words, the mutation sequence $\mathrm{seq}(\boldsymbol{\nu})$ is a shuffle of all such sequences corresponding to the integers $1 \leq k \leq N$. The property that the mutation sequence satisfies the shuffle rule in \deref{Tsyst2} follows from conditions (1) and (2) of \deref{Tsyst2}.  
\end{proof}
%%%%%%%%%%%%%%%%%%%%%%%%%%%%%%%%

\end{document}